\documentclass[11pt,a4paper]{amsart}

\newtheorem{theorem}{Theorem}[section]
\newtheorem{corollary}{Corollary}[section]
\newtheorem{proposition}{Proposition}[section]

\newtheorem{example}[theorem]{Example}

\newtheorem{alg}[theorem]{Algorithm}

\newtheorem{remark}[theorem]{Remark}

\usepackage{amsmath}
\usepackage{amssymb}
\usepackage{latexsym}
\usepackage{tabmac}
\usepackage{amscd}
\usepackage{graphicx,color}

\def \N{\mathbb{N}}
\def \P{\mathbb{P}}

\def\cA{{\mathcal A}}
\def\cB{{\mathcal B}}

\def\cX{{\mathcal X}}
\def\cY{{\mathcal Y}}

\def\rf#1{\left\lceil#1\right\rceil}

\def\({\left(}
\def\){\right)}
\def\[{\left[}
\def\]{\right]}
\def\<{\langle}
\def\>{\rangle}

\def\ts{{\hskip .1in}}

\begin{document}


\title[Conjugation symmetry map]{ Linear time equivalence of
  Littlewood--Richardson coefficient symmetry maps }

\author{Olga Azenhas, Alessandro Conflitti, Ricardo Mamede}
\address{CMUC, Centre for Mathematics, University of Coimbra, Apartado 3008,
3001--454 Coimbra, Portugal.\\
E--mail: \{oazenhas,conflitt,mamede\}@mat.uc.pt
}
\thanks{The three authors are supported by CMUC - Centro de Matem\'atica da
Universidade Coimbra. The second author is also supported by FCT
Portuguese Foundation of Science and Technology (Funda\c{c}\~{a}o
para a Ci\^{e}ncia e a Tecnologia) Grant SFRH/BPD/30471/2006.}


\begin{abstract}
 Benkart, Sottile, and Stroomer have completely
characterized by Knuth  and dual Knuth equivalence  a bijective
proof of the  Littlewood--Richardson coefficient conjugation
symmetry, i.e. $c_{\mu,\nu}^{\lambda}=c_{\mu^t,\nu^t}^{\lambda^t}$.
Tableau--switching provides an algorithm to produce  such a
bijective proof. Fulton has shown that the White and the
Hanlon--Sundaram maps  are versions of that bijection. In this paper
one exhibits explicitly the Yamanouchi word produced by that
conjugation symmetry map which on its turn leads to a new and very
natural version of the same map already considered independently. A
consequence of this latter construction is that using notions of
Relative Computational Complexity we are allowed to show that this
conjugation symmetry map is linear time reducible to the
Sch\"utzenberger  involution and reciprocally. Thus the
Benkart--Sottile--Stroomer conjugation symmetry map with the two
mentioned versions,  the three versions of the commutative symmetry
map, and Sch\"utzenberger involution, are linear time reducible to
each other. This answers a question posed by Pak and Vallejo.

\end{abstract}

\maketitle





\section{Introduction}

Given partitions $\mu$ and $\nu$,  the product $s_{\mu}s_{\nu}$ of
the corresponding Schur functions
 is a non-negative integral
linear combination of Schur functions
$$s_{\mu}s_{\nu}=\sum_{\lambda}
c_{\mu\;\nu}^{\lambda}s_{\lambda},$$ where  $c_{\mu\;\nu}^{\lambda}$
is  called the Littlewood-Richardson coefficient
\cite{lr,mac,sa,st}. Let $\lambda^t$ denote the conjugate or
transpose of the partition $\lambda$.  It is obvious from the
commutativity of multiplication that
$c_{\mu\;\nu}^{\lambda}=c_{\nu\;\mu}^{\lambda}$, called the
commutativity symmetry, and it is less obvious  the conjugation
symmetry $c_{\mu\;\nu}^{\lambda}=c_{\mu^t\;\nu^t}^{\lambda^t}.$ As
there are several Littlewood-Richardson rules to compute these
numbers,
   the combinatorics of their symmetries is quite intriguing  since  in all  of them  the commutativity is hidden,
    and the conjugation is either hidden or partially hidden \cite{bz,knutson0,knutson,pak0}.  This is in contrast with the fact that most
    of the symmetries are explicitly exhibited by simple means \cite{pak0}. By "{\em hidden}" or "{\em simple}" we are referring to the computational complexity  of the operations  needed to reveal such symmetries. Let $\rm
LR(\lambda/\mu,\nu)$ be the set of Littlewood-Richardson (LR for short)
tableaux \cite{lr} of shape $\lambda/\mu$ and content $\nu$. Then
$c_{\mu\;\nu}^{\lambda}$ counts the number of elements of this set.
 If one writes
$c_{\mu\;\nu}^{\lambda}=:c_{\mu\;\nu\,\lambda^\vee}$, with $\lambda^\vee$ the
complement partition of $\lambda$ regarding some rectangle
containing $\lambda$, the
Littlewood-Richardson coefficients
are invariant under the following action of $\mathbb Z_2\times S_3$:
the non--identity element of $\mathbb Z_2$ transposes simultaneously
$\mu$, $\nu$ and $\lambda^\vee$, and $S_3$ sorts $\mu$, $\nu$ and
$\lambda^\vee$ \cite{sottile}.

 The Berenstein-Zelevinsky
interpretation of the Littlewood-Richardson coefficients  (BZ
triangles for short) \cite{bz}
 manifests  all  the $S_3$-symmetries except the commutativity.
  Pak and Vallejo have defined  in \cite{pak0} bijections, which are
explicit linear maps, between  LR tableaux, Knutson-Tao hives
\cite{knutson0} and BZ triangles. These bijections  combined with
the symmetries of BZ triangles  give   all the $S_3$-symmetries
except the commutativity. The conjugation symmetry is also hidden in
BZ triangles.
 In \cite{post}, it is shown  that it can be
revealed from a bijection between web diagrams and BZ-triangles. On
the other hand, the Knutson-Tao-Woodward puzzles \cite{knutson}, the
most symmetrical objects, manifest only
 partially the conjugation symmetry through the
{\em puzzle duality}, {\em viz.}
$c_{\mu\;\nu\;\lambda}=c_{\nu^t\;\mu^t\;\lambda^t}$, since the
commutativity  is hidden. Interestingly, as we shall see,
  a similar {\em partial conjugation symmetry}, $c_{\mu\;\nu\;\lambda}=c_{\lambda^t\;\nu^t\;\mu^t}$,
   is obtained on LR tableaux through a simple bijection, denoted by $\blacklozenge$. In \cite{knutson, king1,king2} bijections  between hives and puzzles can be found.
 Recently, Purbhoo \cite{mosaic} introduced a new
 tool called
  mosaics, a square-triangle-rhombus tiling model with all the rhombi arranged  in the shape of a Young diagram in the corners of an hexagon.
 Mosaics are  in bijection with puzzles and with LR tableaux, and the operation  migration on mosaics,
 which correspond to some sequence of {\em jeu de taquin} operations on LR tableaux, reveals  the  hidden symmetries of puzzles and LR-tableaux.
  The  Thomas-Yong carton rule \cite{carton} is a recent $S_3$-symmetric rule but  the computational complexity of the resulted visual symmetry does not seem
   to be improved as it is based  on non trivial properties of {\em jeu de taquin}.

In \cite{pak},  a number of Young tableau commutative symmetry maps
are considered and it is shown that two of them are linear time
reducible to each other and to the Sch\"utzenberger involution.
(Subsequently  in \cite{DK2} and in \cite{az3} it has been
 shown that the two remaining ones are identical to the others.) In
this paper, we consider three Young tableau conjugation symmetry
maps that appeared in \cite{white, HS,sottile,za,az,az1} and one
shows that these three Young tableau conjugation symmetry maps and
the commutative symmetry maps, considered in \cite{pak}, are linear
time reducible to each other and to the Sch\"utzenberger involution.
In addition, as in the commutative case, the Young tableau
conjugation symmetry maps coincide. This answers a question posed by
Pak and Vallejo in \cite{pak}.

\subsection{Summary of the results}
The conjugation symmetry map is a bijection \cite{pak}
$$\varrho: LR(\lambda/\mu,\nu)\longrightarrow
LR(\lambda^t/\mu^t,\nu^t).$$
Let $T$ be
 a tableau and  $\widehat T$ its standardization.
 The
 Benkart-Sottile-Stroomer conjugation symmetry map \cite{sottile}, denoted by $ \varrho^{BSS}$, is
 the bijection
$$\begin{array}{cccc}
\varrho^{BSS}:LR(\lambda/\mu,\nu)&\longrightarrow&LR(\lambda^t/\mu^t,\nu^t)\cr
T&\mapsto&\varrho^{BSS}(T)=[ Y(\nu^t)]_K\cap [(\widehat{T})^t]_d
\end{array},$$
\noindent where $[ Y(\nu^t)]_K$ is the Knuth class  of all tableaux
with rectification  the Yamanouchi tableau $Y(\nu^t)$ of shape the
conjugate of $\nu$, and $[\widehat{T}^t]_d$ is the dual Knuth class
of all tableaux of shape $\lambda^t/\mu^t$
 with $Q$-symbol the   transpose of $\widehat{T}$. The image of $T$ by the
$BSS$-bijection is the unique tableau  of shape $\lambda^t/\mu^t$ in
both those two equivalence classes. Fulton showed in  \cite{ful}
that the White-Hanlon-Sundaram map $\varrho^{WHS}$ \cite{white,HS}
coincides with $\varrho^{BSS}$. Thus $\varrho^{BSS}(T)$ can be
obtained either by tableau-switching or by the White-Hanlon-Sundaram
transformation $\varrho^{WHS}$.

Given a totally ordered finite alphabet, let $\sigma_i$ denote the
reflection crystal operator acting on a subword over the alphabet
$\{i,i+1\}$, for all $i$ \cite{plaxique,Lot}, and let
$\sigma_0=\sigma_i\cdots\sigma_j\cdots\sigma_k$ be such that
$s_i\cdots s_j\cdots s_k$, with $s_{l}$ the transposition
$\left(l,l+1\right)$, is the longest permutation of $S_{\nu_1^t}$.
The  column reading word of $\varrho^{BSS}(T)$ is the Yamanouchi
word of weight $\nu^t$ whose
 $Q$-symbol  is the one given  by  the column reading word of $\widehat T^t$. The following transformation $\varrho_3$
\cite{za,az,az1,slc} makes clear the construction of that word and
affords a simple way to construct $\varrho^{BSS}(T)$
\def\Tscale{.75}
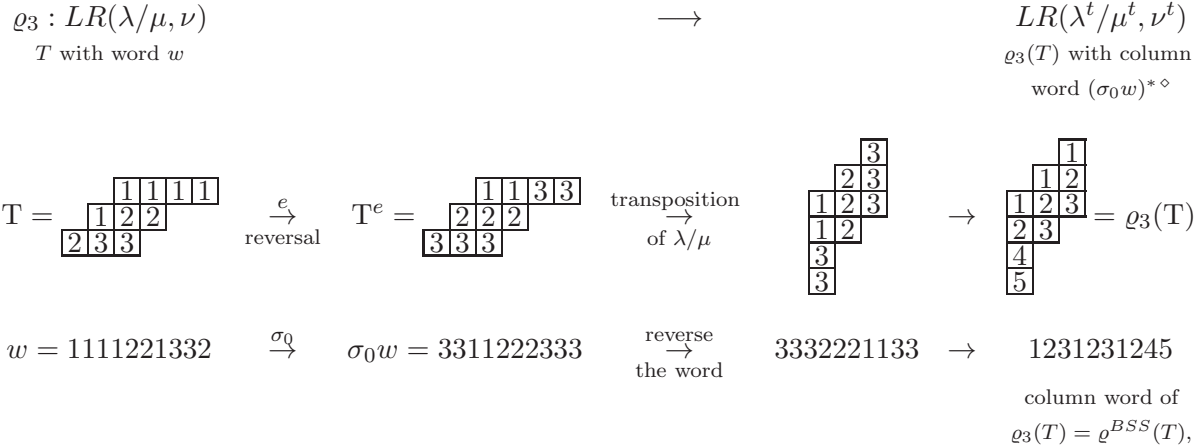
$$
\begin{array}{ccccccccccc}
\varrho_3:LR(\lambda/\mu,\nu)&&&\longrightarrow&&&
LR(\lambda^t/\mu^t,\nu^t)\cr \text{\scriptsize $T$ with  word
$w$}&&&&&& \text{\scriptsize $\varrho_3(T)$ with column  }\cr
 &&&&&&\text{\scriptsize word $(\sigma_0 w)^{*\,\diamond}$}\cr
 &&&&&&\cr
 {{\rm
T}=\tableau[Y]{~,~,1,1,1,1|~,1,2,2|2,3,3}}&
\overset{\text{$e$}}{\underset{\text{\scriptsize
reversal}}\rightarrow}& {{\rm
T}^e=\tableau[Y]{~,~,1,1,3,3|~,2,2,2|3,3,3}}&\underset{\text{\scriptsize
of $\lambda/\mu$}}{\overset{\text{\scriptsize transposition
}}\rightarrow}
&{\tableau[Y]{~,~,3|~,2,3|1,2,3|1,2|3|3}}&\rightarrow&
{\tableau[Y]{~,~,1|~,1,2|1,2,3|2,3|4|5}=\varrho_3(\rm T)}\cr
&&&&&&\cr
 w=1111221332&\overset{\text{\scriptsize$\sigma_0$}}\rightarrow&\sigma_0
w=3311222333&\underset{\text{the
word}}{\overset{\text{reverse}}\rightarrow}&
3332221133&\rightarrow&1231231245\cr &&&&&&\text{\scriptsize column
word of}\cr  &&&&&&\text{\scriptsize
$\varrho_3(T)=\varrho^{BSS}(T)$,}\end{array}$$
 \noindent where
 $*$ denotes the dualization of a word; and
$\diamond$  is the operator which transforms a Yamanouchi word of
weight $\nu$, into a Yamanouchi word of weight $\nu^t $,  by
replacing the subword $i^{\nu_i}$ with $12\dots\nu_i$, for all $i$.
The action of the operator $\diamond$ is extended to dual
Yamanouchi words by putting $(w^*)^{\,\diamond}:= w^{\diamond\,
*}$.  More precisely, the $\diamond$ operator is a
bijection between the Knuth classes of the Yamanouchi tableaux
$Y(\nu)$ and $Y(\nu^t)$, and also between the corresponding  dual
Yamanouchi tableaux. The reversal $e$ of a
 LR tableau can be computed
by the action of  $\sigma_0$ on its word.
 The image of a  LR or dual LR tableau $U$ under rotation of the
skew-diagram by $180 $ degrees, with the dualization $*$  of its
word is denoted by $U^\bullet$; and the image of $U$ under the
 rotation and transposition of the skew-diagram, with
the action of the operation $\diamond$ on its word is denoted by
$U^\blacklozenge$. Again $\bullet$ and $\blacklozenge$ are
involution maps. Then
$$\varrho_3(T)=T^{e\,\bullet\blacklozenge}= T^{\blacklozenge\bullet\,e}=T^{\bullet\blacklozenge\,e}\text{\;  and}$$ $$\text{$(\sigma_0
w)^{*\,\diamond}=(\sigma_0 w)^{\diamond\, *}=\sigma_0( w^{\diamond\,
*})$ is the column word of $T^{e\bullet\blacklozenge}= [
Y(\nu^t)]_K\cap [(\widehat{T})^t]_d$}.$$
In the  two next sections we shall develop the necessary machinery to show the above identities.
 Bijection $\scriptsize{\blacklozenge\bullet}$ appeared
originally in \cite{za} with a different formulation. In
\cite{az,az1} the  bijection $e$, defined differently and based on a
modified insertion, is composed with the last one to give $\rho_3$.
Here we stress the composition of $e\,\bullet$ with $\blacklozenge$.


 Following the ideas introduced in \cite{pak}, we  address, in  Section $4$, the
problem of studying  the computational cost of the  conjugation
symmetry map $\varrho^{BSS}$ utilizing what is known as Relative
Complexity, an approach based on reduction of combinatorial
problems. To this aim we use the version $\varrho_3$.
We consider only linear time reductions; since the bijections we consider
require subquadratic time the reductions have to preserve that.
 Let $\cA$ and $\cB$ be two possibly infinite sets
of finite integer arrays, and let $\delta : \cA \longrightarrow \cB$
be an explicit map between them. We say that $\delta$ has linear
cost if $\delta$ computes $\delta\(A\) \in \cB$ in linear time
$O\(\langle A \rangle \)$ for all $A \in \cA$, where $\langle A
\rangle$ is the bit--size of $A$. The transposition of the recording
matrix of a LR tableau is the recording matrix of a tableau of
normal shape. We have  then a linear map $\tau$ which defines a
bijection between tableaux of normal shape and LR tableaux
\cite{lee1,lee2,pak,ouch}. As the rotation map $\bullet$ and
 $\tau$ are  linear maps, so maps of linear
cost, the  reversal $ T^{e}$ of a LR tableau $T$ can be
linearly reduced  to the evacuation  $E$ of the corresponding
tableau $\tau(T)=P$ of normal shape, {\em i.e.} $\tau(P
^E)=T^{e\,\bullet}$. Additionally, in Algorithm \ref{alg}, it is proved  that the bijection
$\blacklozenge$, exhibiting the symmetry $c_{\mu\;\nu\;\lambda}=c_{\lambda^t\;\nu^t\;\mu^t}$, is  of linear cost.
 The following commutative scheme  shows that the conjugation
symmetry map $\varrho_3$, and therefore $\varrho^{BSS}$ and
$\varrho^{WHS}$, is linear equivalent to the Sch\"utzenberger
involution or evacuation map on tableaux of normal shape,

\begin{theorem}
The following commutative scheme holds
$$\begin{matrix}
T&\overset{\text{$e\,\bullet$}}\longleftrightarrow&T^{e\bullet}&\overset{\text{$\blacklozenge$}}\longleftrightarrow& {T^{e\bullet}}^\blacklozenge\\
{\tau}\updownarrow&&{\tau}\updownarrow\\
P&\underset{E}{\overset{\text{\scriptsize
evacuation}}\longleftrightarrow}&P^E.
\end{matrix}$$
\end{theorem}

\begin{theorem} The conjugation symmetry maps  $\varrho^{WHS}$, $\varrho^{BSS}$ and $\varrho_{3}$ are
identical, and linear time equivalent with the Sch\"utzenberger
involution $E$ and with the reversal map $e$.

\end{theorem}

 We may now extend the list of linear equivalent Young tableau maps established in \cite{pak}, Section $2$,
 Theorem $1$.

\begin{theorem}{\em \cite{pak}}
The following maps are linearly equivalent:

$(1)$  RSK correspondence.

$(2)$  Jeu de taquin map.

$(3)$  Littlewood--Robinson map.

$(4)$  Tableau switching map $s$.

$(5)$  Evacuation (Sch\"utzenberger involution) $E$ for normal
shapes.

$(6)$  Reversal $\,e$.

$(7)$ First  fundamental symmetry map.

$(8)$ Second fundamental symmetry map.
\end{theorem}

\begin{corollary}
The following maps are linearly equivalent:

$(1)$  RSK correspondence.

$(2)$  Jeu de taquin map.

$(3)$  Littlewood--Robinson map.

$(4)$  Tableau switching map $s$.

$(5)$  Evacuation (Sch\"utzenberger involution) $E$ for
normal shapes.

$(6)$  Reversal $\,e$.

$(7)$ First  fundamental symmetry map.

$(8)$ Second fundamental symmetry map.

$(9)$ Third fundamental symmetry map.

$(10)$  $\varrho^{WHS}$ conjugation symmetry map.

$(11)$ $\varrho^{BSS}$ conjugation symmetry map.

$(12)$  $\varrho_3$ conjugation
symmetry map.

In particular,   first and  second fundamental symmetry maps are
identical \cite{DK2}; first and  third fundamental symmetry maps are
identical \cite{az3}; $\varrho^{WHS}$ and  $\varrho^{BSS}$ are
identical conjugation symmetry maps \cite{ful, sottile}, and the
same happens with $\varrho^{BSS}$ and $\varrho_3$.
\end{corollary}


\section{Preliminaries}

\subsection{ Young diagrams and transformations}

A {\it partition} (or normal shape) $\lambda$ is a sequence of
non--negative integers
$\lambda=(\lambda_1,\lambda_2,\ldots,\lambda_{\ell}),$ with
$\lambda_1\geq\lambda_2\geq\cdots\geq\lambda_{\ell}\ge 0$. The
number of parts is $\ell(\lambda)=\ell$  and the weight is
$|\lambda|=\lambda_1+\lambda_2+\cdots+\lambda_{\ell}$. (For
convenience we allow zero parts.) The Young diagram  of $\lambda$ is
the collection of boxes $ \{(i,j)\in\mathbb{Z}^{2}|\; 1\leq
i\leq\ell, 1\leq j\leq \lambda_i\}$. The English convention is
adopted in drawing such a diagram. Throughout the paper we do not
make distinction between a partition $\lambda$ and its Young diagram
\cite{Pak1}. Given partitions $\lambda,\mu$, we say that
$\mu\subseteq\lambda$ if $\mu_i\leq\lambda_i$ for all $i>0$.
 If $(r^{\ell})$ is a $r\times \ell$ rectangle containing
$\lambda$, 
the {\em complement}  of
$\lambda$ regarding that rectangle is the partition
$\lambda^{\vee}=(r-\lambda_{\ell},\dots,r-\lambda_1)$.
We define $\lambda^t$ the
{\em conjugation} or {\em transposition} of $\lambda$ as the  image
of $\lambda$ under the transposition $(i,j)\rightarrow (j,i)$. For
example, let $r=4$ and $\ell=3$. The Young diagram of $\lambda=(3,2,2)$ and its transpose
$\lambda^t=(3,3,1)$
 are  depicted below; and
$\lambda^\vee=(2,2,1)$,
$(\lambda^t)^\vee=(\lambda^{\vee})^t=(3,2,0)$ are depicted by dotted
boxes
\def\Tscale{.9}
$$\begin{matrix}
{\lambda=\tableau[Y]{,,|,|,}}\,,\quad&{\tableau[Y]{,,,\centerdot|,,\centerdot,\centerdot|,,\centerdot,\centerdot|}=\lambda^\vee}\,,\quad&
{\lambda^t=\tableau[Y]{,,|,,||}}\,,\quad&{\tableau[Y]{,,|,,|,\centerdot,\centerdot|\centerdot,\centerdot,\centerdot}={(\lambda^\vee)}^t}.
\end{matrix}$$
A skew-diagram (skew-shape)
$\lambda/\mu$ is $\{(i,j)\in\mathbb{Z}^{2}|\ts (i,j)\in \lambda,
(i,j)\notin\mu\}$ the collection of boxes in $\lambda$ which are not
in $\mu$.
   When $\mu$ is the null partition, the
skew-diagram $\lambda/\mu$ equals the Young diagram  $\lambda$. The
number of boxes in $\lambda/\mu$ is $|\lambda/\mu|=|\lambda|-|\mu|$.
The {\it transpose} (conjugate shape) $(\lambda/\mu)^t$  is the
skew-diagram $\lambda^t/\mu^t$ obtained by transposing the
skew-diagram $\lambda/\mu$.
 Let $r=\lambda_1$. The {\it rotation}
(dual shape) $(\lambda/\mu)^*$  is the image of $\lambda/\mu$ by
rotation of $180$ degrees, or the image of $\lambda/\mu$ under
$(i,j)\longrightarrow (\ell-i+1, r-j+1)$.
 Equivalently
$(\lambda/\mu)^{*}=\mu{}^\vee/\lambda^\vee$. In particular,
$\lambda^*$ is the skew-diagram $r^\ell/\lambda^\vee$. The dual
conjugate shape $(\lambda/\mu)^\diamondsuit$ is the image of
$\lambda/\mu$  under $(i,j)\longrightarrow (r-j+1,\ell-i+1)$. The map
$\diamondsuit$ is the composition of the transposition with the
rotation maps $\diamondsuit=* t=t *$. In particular,
$\lambda^\diamondsuit=\ell^r/(\lambda^\vee)^t$. For instance, if
$\mu=(2)\subset\lambda=(4,3,1)$, we have

\noindent\def\Tscale{.9}
$\begin{array}{ccccccc}
{\text{$\lambda/\mu$}=\tableau[Y]{~,~,~,,|~,,,|~,|}}&{ \text{
$(\lambda/\mu)^t$}=\tableau[Y]{~,~,,|~,~,|~,,|~,|}}& {\ts\text{
\ts$(\lambda/\mu)^*$}=\tableau[Y]{~,~,~,~,|~,~,,,|~,,}}&
{\text{\ts$(\lambda/\mu)^\diamondsuit$}=\tableau[Y]{~,~,~,|~,~,,|~,~,|~,,|}}.
\end{array}$

\subsection{Tableaux and words}

The Littlewood--Richardson (LR for short) numbering (reading) of the
boxes of a skew-diagram $\lambda/\mu$ is an assignment of the labels
$1,2,\ldots$ which sorts the boxes of $\lambda/\mu$ in increasing
order from right to left along each row, starting in the top row and
moving downwards; and the column LR numbering of the boxes sorts in
increasing order, from right to left along each column, starting in
the rightmost column and moving downwards. Analogously the reverse
LR numbering  and the column LR numbering of $\lambda/\mu$ are
defined.

\begin{example} If $\lambda/\mu=\tableau[Y]{~,~,,|,,,}$, the
LR-numbering, column LR-numbering and the corresponding reverse
LR-numberings of $\lambda/\mu$ are, respectively,
$$\begin{array}{ccccccccc}{\,\,\setlength{\arraycolsep}{0.1cm}\renewcommand{\arraystretch}{0.5}\begin{matrix}
  & &2&1\\
6&5&4&3\\
\end{matrix}}&&{\,\,\setlength{\arraycolsep}{0.1cm}\renewcommand{\arraystretch}{0.5}\begin{matrix}
    &&3&1\\
6&5&4&2\\
\end{matrix}}&&
{\,\,\setlength{\arraycolsep}{0.1cm}\renewcommand{\arraystretch}{0.5}\begin{matrix}
&&5&6\\
1&2&3&4\\
\end{matrix}}&&{\setlength{\arraycolsep}{0.1cm}\renewcommand{\arraystretch}{0.5}\begin{matrix}
  & &4&6\\
1&2&3&5\\
\end{matrix}}.\end{array}$$

\end{example}
Clearly, the column LR-numbering of $\lambda/\mu$ is the
LR-numbering of $(\lambda/\mu)^\diamondsuit$, and the reverses of
LR-numbering and column LR-numbering of $\lambda/\mu$ are,
respectively, the LR-numbering of ${(\lambda/\mu)}^*$ and
${(\lambda/\mu)}^t$.

 A {\it Young tableau} $ T$ of shape $\lambda/\mu$ is  a filling
of the boxes of the  skew-diagram $\lambda/\mu$ with positive
integers in $\{1,\dots,t\}$ which is increasing in columns from top
to bottom and non-decreasing in rows from left to right. When $\mu$
is the empty partition we say that $ T$ has normal shape $\lambda$.
 The word $w( T)$ of a Young tableau $ T$ is the
sequence obtained by reading the entries of $ T$ according to its
LR numbering, that is, reading right-to-left the rows of $ T$,
from top to bottom. The column word $w_{col}(T)$ is the word
obtained according the column LR numbering. The weight of $T$ is the
weight of of its word.
Denote by $ YT(\lambda/\mu,m)$ the set of Young tableaux of shape
$\lambda/\mu$ and weight $m=(m_1,\ldots,m_t)$.
\begin{example}\label{yamanouchi} $
T=\tableau[Y]{~,~,1,1,1,1|~,1,2,2|2,3,3}$, $w(T)=1111221332$ and
$w_{col}(T)=1112123132$.
\end{example}
A Young tableau with $s $ boxes is {\it standard} if it is filled
with $\{1,\dots,s\}$  without repetitions. Given a tableau $
T$ of weight $m$, the {\it standardization} of $T$, denoted by $
\widehat{T}$,  
is obtained by replacing, west to east, the letters $1$ in $\rm T$
with $1,2,\ldots,m_1$; the letters $2$  with $m_1+1,\ldots,m_1+m_2$;
and so on.  The standardization $\widehat w$ of a word $w$ is
defined accordingly, from right to left. For instance, the
standardization of the tableau $ T$ in the previous example is $
\widehat{T}=\tableau[Y]{~,~,2,3,4,5|~,1,7,8|6,9,10},$  and
$\widehat{ w(T)}:=w(\widehat T)=54328711096$. If $w=w_1 w_2\dots
w_s$ is a word  and $\alpha$ is a permutation in the symmetric group
$\mathcal{S}_{s}$, define $\alpha w= w_{\alpha(1)}\dots
w_{\alpha(s)}$. In the case $T$ is standard we have
$w_{col}(\widehat T)= rev\, w({\widehat T}^t)$, with $rev$ the
reverse permutation.

A Young tableau $ T$ is said a {\it Littlewood--Richardson} (LR
for short) tableau if its word, when read from the beginning to any
letter, contains at least as many letters $i$ as letters $i+1$, for
all $i$. More generally, a word such that every prefix satisfies this
property is called a {\it lattice permutation} or  a Yamanouchi
word. Notice that the column word of a LR--tableau is also a
Yamanouchi word of the same weight.
Denote by ${ LR}(\lambda/\mu,\nu)$ the set of LR tableaux of
shape $\lambda/\mu$ and weight $\nu$. When $\mu=0$ we get the
Yamanouchi tableau $Y(\nu)$,  the unique tableau of shape and weight
$\nu$.  In Example \ref{yamanouchi}, $T$  is a LR tableau with
Yamanouchi word $w(T)=1111221332$ and  column word
$w_{col}(T)=1112123132$.

There is an one--to--one correspondence between Yamanouchi words of
weight $\nu$ and standard tableaux of shape $\nu$. Let $w=w_1
w_2\cdots w_s$ be  a Yamanouchi word and put the number $k$ in
the $w_k$th row of the diagram $\nu$. The  labels of the $i$th row
are the $k$'s such that $w_k=i$, thus the length is  $\nu_i$ and the
shape is $\nu$. We denote this standard tableau by $U(w)$.
 In Example \ref{yamanouchi},  $w=1111221332$,
 and $U(w)=\tableau[Y]{1,2,3,4,7|5,6,10|8,9}$ where the entries of the $i$th row are the positions  of the $i$'s in the LR reading of $T$ .

\subsection{Matrices and tableaux}

Given $T\in Y(\lambda/\mu,m)$,
let $M=(M_{ij})_{1\le i\le \ell(\lambda),1\le j\le t}$ be a matrix
with non--negative entries such that
$M_{ij}$ is the number of $j's$ in the $i$th row of $T$, called
the recording matrix of $T$ \cite{lee1,lee2,pak}.
 The recording matrix of a tableau of normal shape is an upper
triangular matrix, and the recording matrix of an LR tableau is a
lower triangular matrix. Thus we have an one--to--one correspondence
between LR tableaux and tableaux of normal shape as follows.
Considering $T$ in Example \ref{yamanouchi}, the recording matrix of
$T$ is $M=\left(\begin{array}{cccccccc}
4&0&0\\
1&2&0\\
0&1&2
\end{array}\right)$. On the other hand, the transposition  $M^t=\left(\begin{array}{cccccccc}
4&1&0\\
0&2&1\\
0&0&2
\end{array}\right)$ encodes the tableau
 $B=\tableau[Y]{1,1,1,1,2|2,2,3|3,3}$ of normal shape $\nu$ and
weight $\lambda-\mu$.  For two Young diagrams $\mu$ and $\nu$,
define  $\nu\circ
\mu=(\nu_1+\mu_1,\ldots,\nu_1+\mu_\ell,\nu_1,\ldots,\nu_r)/(\mu_1+\nu_1,\ldots,\mu_\ell+\nu_1)$,
$\ell=\ell(\mu)$, $r=\ell(\nu)$. Then with $\mu=(1)$, $B\circ
Y(\mu)=\tableau[Y]{~,~,~,~,~,1|1,1,1,1,2|2,2,3|3,3}\in
LR(\nu\circ\mu,\lambda).$ Given partitions $\lambda$, $\mu$, $\nu$
such that
 $|\lambda|=|\mu|+|\nu|$, define
 $CF(\nu,\mu,\lambda)=\{B\in YT(\nu,\lambda-\mu):B\circ Y(\mu)\in
 \rm LR(\nu\circ\mu,\lambda)\}$ \cite{pak}.
  The map $\tau:LR(\lambda/\mu, \nu)\rightarrow CF(\nu,\mu,\lambda)$
  such that $\tau(M)$ is the tableau of normal shape with recording
  matrix $M^t$, where $ M$ is the recording matrix of $T$, is a bijection. Taking
  again Example \ref{yamanouchi}, we have $\tau(T)=B$.

\subsection{Rotation and transposition of LR tableaux }
Given an integer $i$ in $\{1,\ldots,t\}$,  let
$i^*:=t-i+1$.
Given a word $w=w_1w_2\cdots w_{s}$, over the alphabet
$\{1,\ldots,t\}$, of weight $m=(m_1,\ldots,m_t)$,
$w^*:=w_{s}^*\cdots w_2^*w_1^*$ is the {\em dual word}
 of $w$ and  $m^*=(m_t,\ldots,m_1)$  its weight.   Indeed $w^{**}=w$.
 A dual Yamanouchi word is a word whose dual word is Yamanouchi.
  Given a Young tableau $\rm T$ of shape $\lambda/\mu$ and weight
$(m_1,\ldots,m_t)$,   ${\rm T}^\bullet$ denotes the Young tableau of
shape $(\lambda/\mu)^*$ and weight $m^*$, obtained from $\rm T$ by
replacing each entry $i$ with  $i^*$, and then rotating the result
by 180 degrees. The word of ${\rm T^\bullet}$ is $w(\rm T)^*$, and
${\rm T}^{\bullet\bullet}=\rm T$.
 A dual LR tableau is a
tableau whose word is a dual Yamanouchi word. $\rm
LR(\lambda/\mu,\nu^*)$ denotes the set
 of dual LR tableaux of shape $\lambda/\mu$ and weight
$\nu^*$, and  is the image of $\rm LR((\lambda/\mu)^*,\nu)$ under
the rotation map $\bullet$. Thus the rotation map $\bullet$ defines
a bijection between $\rm LR((\lambda/\mu)^*,\nu^*)$ and $\rm
LR(\lambda/\mu,\nu)$. Given a Yamanouchi word $w$ of weight $\nu$,
define the standard tableau
 $U(w^*)$ of  shape $\nu^*$ such that the label $k$ is in
row $i$ if and only if $w_{s-i+1}=k^*$. Thus $U(w^*)=U(w)^\bullet$
and this affords a bijection between dual Yamanouchi words of weight
$\nu^*$ and  standard tableaux  of  shape $\nu^*$. The rotation map
$\bullet$ is a linear map:
 $M=(M_{ij})$ is the
recording matrix of $T$ if and only if  the recording matrix of
$T^\bullet$ is $(M_{s+1-i,t-j+1})$.

There is another natural bijection, denoted by $\blacklozenge$,
between LR tableaux  of conjugate weight and   dual conjugate shape,
 see  \cite{za,az,az1}. Given a Yamanouchi word $w$ of weight
$\nu=(\nu_1,\dots,\nu_t)$, write $\nu^t=(\nu_1^t,\dots,\nu_k^t)$ and
observe that $w$ is a shuffle of the words $12\dots\nu^t_i$ for all
$i$, and its dual word is a shuffle of the words $t\,t-1\,\cdots
t-\nu_i^t+1$, for all $i$. Thus, we define $w^\diamond$ as the
Yamanouchi word of weight $\nu^t$ obtained by replacing the subword
 consisting only on the letters $i$
     with the subword $12\cdots \nu_i$, for each $i$. The operation $\diamond$ is defined
     on dual Yamanouchi words  by ${w^*}^\diamond:={w^\diamond}^*=$, giving rise to a dual Yamanouchi word of weight $\nu^{*t}$.
     The word ${w^\diamond}^*$ can be obtained in just only one
     step: replace the subword of $w$
 consisting only on the letters $i$
     with the subword $\nu_1\,\nu_1-1\cdots \nu_1-\nu_i+1$, for all
     $i$.
Clearly,  $U(w^\diamond)=U(w)^t$ is of shape $\nu^t$, and
      $U(w^{*\diamond})=U(w)^{\bullet t}$ is of shape $\nu^*$.
Given ${\rm T}\in {\rm LR}(\lambda/\mu,\nu)$ (${\rm
LR}(\lambda/\mu,\nu^*)$) with word $w$, define ${\rm
T}^{\blacklozenge}$ as the LR tableau of shape
$(\lambda/\mu)^\diamondsuit$ and weight $\nu^t$ obtained from $\rm
T$ by replacing the word $w$ with $w^{\diamond}$, and then rotating
the result by 180 degrees and transposing. Then
${\scriptsize{\blacklozenge}}:{\rm LR}(\lambda/\mu,\nu) (  {\rm
LR}(\lambda/\mu,\nu^*))\longrightarrow {\rm
LR}((\lambda/\mu)^{\diamondsuit},\nu^t){\rm
LR}((\lambda/\mu)^{\diamondsuit},\nu^{*t})$ is a bijection such that
$T^{\scriptsize{\blacklozenge}}$ has  column word $w^\diamond$ and
 $\rm T^{\scriptsize{\blacklozenge\blacklozenge}}=T$. Since
$\scriptsize{\blacklozenge\bullet=\bullet\blacklozenge}$,
$T^{\blacklozenge\bullet}=T^{ \bullet\blacklozenge}\in{\rm
LR}((\lambda/\mu)^t,\nu^{t*})$ (${\rm
LR}((\lambda/\mu)^{t},\nu^{t})$) has column word $w^{*\diamond}$.

\begin{example} $\rm
    T=\tableau[Y]{~,~,1,1|~,1,2,2|1,3}$ is a LR tableau
   with word $w=1122131$ of weight $\nu=(4,2,1)$. Then
    ${\rm
    T}=\tableau[Y]{~,~,1,1|~,1,2,2|1,3}\ts\overset{\text{replace }w({\rm T})}{\underset{\text{by }w({\rm T})^{\diamond}}
    {\longleftrightarrow}}\ts
    \tableau[Y]{~,~,2,1|~,3,2,1|4,1}\ts\overset{\text{rotate}}{\underset{\text{transpose}}{\longleftrightarrow}}\ts
    \tableau[Y]{~,1,1|~,2,2|1,3|4}={\rm T}^{\blacklozenge}{\underset{\text{$\bullet$}}{\longleftrightarrow}}\ts
    \tableau[Y]{~,~,1|~,2,4|3,3|4,4}={\rm T}^{\blacklozenge\bullet}\;.$
$\rm T^{\blacklozenge}$ is a LR tableau with
 shape $(\lambda/\mu)^{\diamondsuit}$ and column word $w^\diamond=1212314$ of weight $\nu^t$.
${\rm T}^{\blacklozenge\bullet}$ is a dual LR tableau with
 shape $(\lambda/\mu)^{t}$  and column word $w^{\diamond
 *}=1423434$ of weight $\nu^{t*}$,
  where $U(w)=\tableau[Y]{1,2,5,7|3,4|6}$,
  $U(w^\diamond)=\tableau[Y]{1,3,6|2,4|5|7}=U(w)^t$, and $U(w^{\diamond *})=\tableau[Y]{~,~,1|~,~,3|~,4,6|2,5,7}=U(w)^{\bullet t}$.

\end{example}


\section{Conjugation symmetry maps}
\subsection{Knuth equivalence and dual Knuth equivalence}

Whenever partitions $\nu\subset\mu\subset\lambda$, we say that
    $\lambda/\mu$ extends $\mu/\nu$. An {\it inside corner} of
    $\lambda/\mu$ is a box in  the diagram $\mu$ such that the boxes below and to the right are not in $\mu$.
    When a box extends $\lambda/\mu$, this box is called an
    {\it outside corner}.
Let $\rm T$ be a Young tableau and let $b$ be an inside corner for $
T$. A {\it contracting slide} \cite{schu,sottile} of $ T$ into the
box $b$ is performed by moving the empty box at $b$ through $\rm T$,
successively interchanging it with the neighboring integers to the
south and east according to the following rules:
$(i)$ if the empty box has only one neighbor, interchange with that
neighbor;
    $(ii)$ if it has two unequal neighbors, interchange with the
    smaller one; and
    $(iii)$ if it has two equal neighbors, interchange with that one to
    the south.
The empty box moves in this fashion until it has become an outside
corner. This contracting slide can be reversed by performing an
analogous procedure over the outside corner, called an {\it
expanding slide}. Performing a contracting slide over each inside
corner of $ T$ reduces $ T$ to a tableau ${ T}^{\rm n}$ of
normal shape. This procedure is  known as {\it jeu de taquin}.
 ${ T}^{\rm n}$ is independent of the
particular sequence of inside corners used \cite{thomas}, and so
${\rm T}^{\rm n}$ is  called the {\it rectification} of $\rm T$. A
word $w$ corresponds by RSK--correspondence to a pair $(P(w), Q(w))$
 of tableaux of the same shape, with $Q(w)$ standard, called the $Q$--symbol or recording tableau of
 $w$.
  Here we consider a variation of RSK-correspondence known as
 the  Burge   correspondence \cite{burge,ful}.
Given $w=w_1w_2\cdots w_s$,  ${P}(w)$ is the insertion tableau
obtained by column insertion of the letters
 of $w$ from left to right \cite{ful}. The corresponding recording
tableau ${Q}(w)$ is obtained by placing in $1,2,\dots ,s $. If $w$
is the word of  $\rm T$ then $P(w)={ T}^{\rm n}$. Insertion can be
translated into the language of Knuth elementary transformations.
Two words $w$ and $v$ are said Knuth equivalent if they have the
same insertion tableau. Each Knuth class is in bijection with the
set of standard tableaux with  shape equal to the unique tableau in
that class. Two tableaux $ T$ and $ R$ are {\it Knuth equivalent},
written $ T\equiv R$, if and only if $P(w(T))=P(w(R))$.
Equivalently, $T^{\rm n}= R^{\rm n}$, i.e. one of them can be
transformed into the other one by a sequence of {\em jeu de taquin}
slides. The insertion tableau of a  Yamanouchi word $w$ with
partition weight $\nu$, is the Yamanouchi tableau ${ Y}(\nu)$. The
recording tableau of a Yamanouchi word $w$ is $U(w)$.

Two tableaux $ T$ and $ R$ of the same shape are {\it dual
equivalent}, written ${ T}\overset{d}{\equiv}{ R}$, if any sequence
of contracting slides and expanding slides that can be applied to
one of them, can also be applied to the other, and the sequence of
shape changes is the same for both \cite{h2,ful}.
 Dual equivalence may
also be characterized by recording tableaux:
 ${ T}\overset{d}{\equiv}{
    R}$ if and only if ${Q}(w({ T}))={Q}(w({ R}))$.
Thus two tableaux of the same normal shape are dual equivalent. Let
$ S$ and $ T$ be tableaux such that $\rm T$ extends $\rm S$, and
consider the set union $\rm S\cup T$. The {\em tableau switching }
\cite{sottile} is a procedure  based on {\em jeu de taquin}
elementary moves on two alphabets that transforms $\rm S\cup \rm T$
into $\rm A\cup \rm B$, where $B$ is a tableau Knuth equivalent to $
T$ which extends $\rm A$, and $A$ is a tableau Knuth equivalent to $
S$. We write $\rm S\cup \rm T\overset {s}\longrightarrow \rm A\cup
\rm B$. In particular, if  $ S$ is of normal shape,  $\rm A=\rm
T^{\rm n}$, and $ S= B^{\rm n}$. Switching of $S$ with $T$ may be
described as follows:  $\widehat T$ is a set of instructions telling
where expanding slides can be applied to $S$. Thus switching and
dual equivalence are related as below and tableaux are completely
characterized by dual and Knuth equivalence.
\begin{theorem}{\em \cite{h2}}\label{t1}
    Let $ T$ and $ U$ be  tableaux with the same normal shape and let $ W$ be a tableau
     which extends  $ T$.
$(1)$ If $ T\cup
      W\overset{s}\longrightarrow  Z\cup  X  $ and $ U\cup  W\overset{s}\longrightarrow  Z\cup
     Y$, then $ X\overset{d}\equiv  Y$.

     $(2)$ Let $\mathcal{D}$ be a dual equivalence class and $\mathcal{K}$
    be a Knuth equivalence class, both corresponding to the same normal shape.
    Then, there is a unique tableau in $\mathcal{D}\cap\mathcal{K}$.

\end{theorem}

 Algorithm to construct ${\mathcal D}\cap {\mathcal K}$: Let
$U\in {\mathcal D}$ and let $V\in {\mathcal K}$ be the only tableau
with
 normal shape in this class, and $W$ any tableau that $U$ extends:
$\begin{matrix}
W\cup U &&W\cup X\\
{\scriptstyle s}\downarrow&&\uparrow{\scriptstyle s}\\
U^{\rm n}\cup Z&\rightarrow&V\cup Z.
\end{matrix}$ Thus $X\overset{d}{\equiv} U$,
$X\overset{}{\equiv}V$,
and  $\mathcal{D}\cap\mathcal{K}=\{X\}$.
 since two words in the same Knuth class can
 not have the same $Q$--symbol.

\subsection{The transposition of the rotated  reversal LR tableau }

    Given a tableau $\rm T$ of normal shape, the evacuation ${\rm T}^E$ is the rectification of $\rm T^\bullet$, that
    is,
    $\rm T^E={\rm T^\bullet}^n$.
${\rm T}^E$ is also  obtained
 either as the insertion tableau of the word $w({\rm
 T})^*$; or  according to
 the Sch\"utzenberger evacuation algorithm; or
  applying the reverse {\em jeu de taquin} slides to $T$, in the smallest rectangle containing $T$, to obtain
$T^{\rm a}$ the anti-normal form $T$. Thus ${T^{\rm a}
}^\bullet=T^E=T^{\bullet \rm n}$. If $w$ is a Yamanouchi word, by
duality of Burge correspondence, $Q(w^*)=U(w)^E=U(w)^{\bullet \rm
n}={U(w)^{\rm a} }^\bullet$.
Given $w\equiv{\rm Y}(\nu)$, we may now define the word
$w^{\diamond}$ as being the unique word satisfying
$w^{\diamond}\equiv{\rm Y}(\nu^t)$ such that ${Q}(w^\diamond)
={Q}(w)^T=U(w)^t$. Since
     $(w^{\diamond})^{\diamond}=w$, the map $w\mapsto w^{\diamond}$ establishes a bijection
     between the Knuth classes of ${ Y}(\nu)$ and ${
     Y}(\nu^t)$. The word $w^*$ is the
     unique word satisfying $w^{\*}\equiv{\rm Y}(\nu^*)$ such that
${Q}(w^*) = U(w)^{\bullet\rm n}$,  and $w^{\diamond*}$ is the unique
word satisfying $w^{\diamond*}\equiv{\rm Y}(\nu^{t\,*})$ such that
$Q(w^{\diamond*})=U(w)^{E\,t}$.

    Given a tableau ${ T}$ of any shape,  the reversal ${ T}^e$ is the unique tableau  Knuth equivalent to ${\rm T}^\bullet$, and
    dual equivalent to ${\rm T}$ \cite{sottile}.
 By Theorem~\ref{t1},
$T^e=[T^{\rm n\,E}]_K\cap[T]_d $, where $[$ $ ]_K$ denotes Knuth
class and $[$ $ ]_d$ dual Knuth class. If $T$ has normal shape,
$T^E=T^e$. If ${ T}\in{\rm LR}(\lambda/\mu,\nu)$, then ${ T}^e$ is
the only tableau Knuth equivalent to ${ Y}(\nu^*)$ and dual
equivalent to $ T$. Since crystal reflection operators, for the
definition see \cite{plaxique,Lot}, preserve the $Q$--symbol, we may
in the case of LR tableaux characterize explicitly the word of $T^e$
as follows. Let $w$ be a Yamanouchi word of weight
$\nu=(\nu_1,\dots,\nu_t)$, and let $\sigma_i$ denote the reflection
crystal  operator
 acting on the subword
 over the alphabet $\{i,i+1\}$, for all $i$.
If $s_{i_r}\cdots s_{i_1}$ is the longest permutation in
$\mathcal{S}_t$, put $\sigma_0:=\sigma_{i_r}\cdots\sigma_{i_1}$.
Then $\sigma_0 w$ is a dual Yamanouchi word of weight $\nu^*$.
Moreover, $w\equiv w'$ if and only if
$\sigma_i(w)\equiv\sigma_i(w')$, and ${Q}(w)={Q}(\sigma_i(w))$.
Thus, we have proven the following

\begin{theorem}\label{plr*}
    Let $\rm T$ be a LR tableau with
shape $\lambda/\mu$ and word $w$. Then ${\rm T}^e$ is the dual LR
tableau of shape $\lambda/\mu$ and word $\sigma_0w$, and
$T^{e\,\blacklozenge\bullet}$ is the LR tableau of shape
$(\lambda/\mu)^t$ and column word $(\sigma_0 w)^{\diamond\,*}$.
\end{theorem}

\begin{corollary} $T^{e\,\blacklozenge\bullet}$ is the unique tableau
Knuth equivalent to $Y(\nu^t)$ and dual equivalent to $\widehat
T^t$.
\end{corollary}
\begin{proof} It is enough to see that the column words of $T^{e\,\blacklozenge\bullet}$ and $\widehat T^t$ have the same $Q$--symbol.
Let $\widehat w$ be the word of $\widehat T$. As $rev\,\widehat w$,
the reverse word  of
 $\widehat T$, is the column word of  $\widehat T^t$, then
 $Q(rev\,\widehat w)=Q(\widehat w)^{E\,t}=$ $Q(w)^{E\,t}=Q(w^{\diamond\,
 *})=$ $Q(\sigma_0(w^{\diamond\, *}))=Q((\sigma_0w)^{\diamond\,
 *}).$
\end{proof}

We recall that the action of crystal reflection operators on
 words corresponds to {\em jeu de taquin} slides on
two-row tableaux. In particular,
 if $w$ is a Yamanouchi word of weight $\nu$ and
  $\theta _i$ denotes the {\em jeu de taquin} action on the
consecutive rows $i$ and $i+1$ of $U(w)$, then $\theta_iU(w)$  is a
tableau of skew-shape $(i\,i+1)\nu$ such that  any two consecutive
rows define a two-row tableau of normal or anti-normal shape. The
labels of the $j$-th row of $\theta_iU(w)$ are precisely the $k$'s
such that $(\sigma_i w)_ k = j$. Put
$\theta_0:=\theta_{i_r}\dots\theta_{i_1}$ with $i_r,\dots,i_1$ as in
$\sigma_0$. Thus $\theta_0U(w)=U(w)^{\rm a }$ and $Q(\sigma_0
w)=U(w)^{\rm a \,\rm n}$. This defines the commutative scheme
$$\begin{matrix}
w&\longleftrightarrow&\sigma_{i_1}w&\longleftrightarrow&\sigma_{i_2}\sigma_{i_1}w&\longleftrightarrow&\cdots&\longleftrightarrow&\sigma_0w\\
\updownarrow&        &\updownarrow&                    &\updownarrow           &                    &      &                  &\updownarrow\\
U(w)&\longleftrightarrow&\theta_{i_1}U(w)&\longleftrightarrow&\theta_{i_2}\theta_{i_1}U(
w)&\longleftrightarrow&\cdots&\longleftrightarrow&\theta_0U(w).
\end{matrix}$$

\noindent(This was the procedure in \cite{az}.)
 Similarly, if $\sigma_iT$ denotes the tableau obtained by the action of $\sigma_i$ on its word, we
 get the commutative scheme
$$\begin{matrix}
T&\longleftrightarrow&\sigma_{i_1}T&\longleftrightarrow&\sigma_{i_2}\sigma_{i_1}T&\longleftrightarrow&\cdots&\longleftrightarrow&\sigma_0T\\
{\tau}\updownarrow&        &{\tau}\updownarrow&                    &{\tau}\updownarrow           &                    &      &                  &{\tau}\updownarrow\\
P&\longleftrightarrow&\theta_{i_1}P&\longleftrightarrow&\theta_{i_2}\theta_{i_1}P&\longleftrightarrow&\cdots&\longleftrightarrow&\theta_0P=P^{\rm
a}.\end{matrix}$$

\begin{theorem}\label{reduc} Let $T$ be a LR tableau and $\tau (T)=P$. Then the
following commutative scheme holds
$$\begin{matrix}
T&\overset{\text{$e$}}\longleftrightarrow&T^e&\overset{\text{\,$\bullet$}}\longleftrightarrow&T^{e\bullet}
\\
{\tau}\updownarrow&&{\tau}\updownarrow&&{\tau}\updownarrow\\P&\underset{\text{$\rm
a$}}{\overset{\text{\scriptsize }}\longleftrightarrow}&P^{\rm
a}&\overset{\text{$\bullet$}}\longleftrightarrow &P^E.
\end{matrix}$$
\end{theorem}


\subsection{ Main bijections} As already mentioned bijections $\varrho^{WHS}$ and
$\varrho^{BSS}$ are identical. We now focus on $\varrho^{BSS}$ and
$\varrho_3$.
 Let
$$\begin{matrix}
\varrho^{BSS}:&LR(\lambda/\mu,\nu)&\rightarrow&LR(\lambda^t/\mu^t,\nu^t)\\
&T&\mapsto&\varrho^{BSS}(T)=[Y(\nu^t)]_K\cap[\widehat{T}^t]_d
\end{matrix}\qquad \text{\cite{sottile}}.$$
\noindent  The image of $T$ by the $BSS$-bijection is the unique
tableau  of shape $\lambda^t/\mu^t$ whose rectification is
$Y(\nu^t)$ and the $Q$--symbol of the column reading word is
$Q(T)^{Et}$. The idea behind this bijection can be told as follows:
$\widehat T$ constitutes a set of instructions telling where
expanding slides can be applied to $Y(\mu)$. Then $\widehat T^t$ is
a set of instructions telling where expanding slides can be applied
to $Y(\mu)^t$. Tableau--switching provides an algorithm  to give way
to those instructions: $$\begin{array}{ccccccccccccccc} {Y(\mu)\cup
T}&\underset{\text{of T}}{\overset{\text{\scriptsize
standardization}}{\longrightarrow}}& {Y(\mu)\cup \widehat
T}&\underset{\text{of $\widehat T$}}{\overset{\text{\scriptsize
transposition}}{\longrightarrow}}&{Y(\mu^t)\cup {\widehat T}^t}&&Y(\mu^t)\cup \varrho^{BSS}(T)&\\
&&&&
 \downarrow{\scriptstyle s}&&\uparrow{\scriptstyle s}&\\
 &&&&{(\widehat{T}^t)^{\rm n}\cup Z}&\overset{{}}\mapsto&Y(\nu^t)\cup Z&
\end{array}.$$
\noindent 
Then
$\varrho^{BSS}(T)\equiv Y(\nu^t)$ and $\varrho^{BSS}(T)\equiv_d
\widehat T^t$.

\begin{example} Let $T$ in  $LR (\lambda/\mu, \nu)$ with $\mu=(2,1)$, $\nu=(5,3,2)$ and $\lambda=(6,4,3):$

$\begin{array}{cccccccccc}
{T=\tableau[Y]{~,~,1,1,1,1|~,1,2,2|2,3,3}}& \rightarrow&
{\widehat{T}=\tableau[Y]{~,~,2,3,4,5|~,1,7,8|6,9,10}}& \rightarrow&
&{\widehat{T}^t=\tableau[Y]{~,~,6|~,1,9|2,7,10|3,8|4|5}}&\rightarrow
\end{array}$

${\begin{matrix} \rightarrow
Y(\mu^t)\cup\widehat{T}^t=
{\tableau[Y]{\color{red}{\bf{1}},{\color{red}\bf{1}},6|{\color{red}\bf{2}},1,9|2,7,10|3,8|4|5}}
&&
{\tableau[Y]{{\color{red}\bf{1}},{\color{red}\bf{1}},1|{\color{red}\bf{2}},1,2|1,2,3|2,3|4|5}}=Y(\nu^t)\cup\varrho^{BSS}(T)\\
{\scriptstyle s}\downarrow&&\uparrow{\scriptstyle s}\\
(\widehat{T}^t)^{\rm n}\cup
Z={\tableau[Y]{1,6,9|2,7,10|3,8,{\color{red}\bf{1}}|4,{\color{red}\bf{2}}|5|{\color{red}\bf{1}}}}&\;\;\overset{\text{}}\longrightarrow\;\;
&{\tableau[Y]{1,1,1|2,2,2|3,3,{\color{red}\bf{1}}|4,{\color{red}\bf{2}}|5|{\color{red}\bf{1}}}}=Y(\nu^t)\cup
Z
\end{matrix}}.$
\end{example}

Let
$$\begin{matrix}
\varrho_3:&LR(\lambda/\mu,\nu)&\rightarrow&LR(\lambda^t/\mu^t,\nu^t)\\
&T&\mapsto&\varrho_3(T)=T^{e\,\bullet\,\blacklozenge}\\
&w&\mapsto&\sigma_0 w^{*\diamond}
\end{matrix}\qquad \text{\cite{za,az,az1}}.$$
\noindent As $T^{e\bullet\blacklozenge}$ is  the unique tableau
Knuth
 equivalent to $Y( \nu^t)$
 and dual equivalent to $(\widehat{T})^t$, we have

\begin{corollary} $\varrho^{BSS}$
and $\varrho_3$ are identical bijections.

\end{corollary}

\begin{example} Let $T$ in  $LR (\lambda/\mu, \nu)$ as before:
\def\Tscale{.78}
$$
\begin{array}{ccccccccccc}
 {{\rm
T}=\tableau[Y]{~,~,1,1,1,1|~,1,2,2|2,3,3}}&\overset{\text{$e$}}{\underset{\text{reversal}}{\rightarrow}}&
{{\rm
T}^e=\tableau[Y]{~,~,1,1,3,3|~,2,2,2|3,3,3}}&\underset{\text{of
$\lambda/\mu$}}{\overset{\text{transpose}}{\rightarrow}}
&{\tableau[Y]{~,~,3|~,2,3|1,2,3|1,2|3|3}}&\rightarrow&
{\tableau[Y]{~,~,1|~,1,2|1,2,3|2,3|4|5}=\varrho^{BSS}(\rm T)}\cr
&&&&&&\cr
 w=1111221332&\rightarrow&\sigma_0
w=3311222333&\overset{\text{reverse}}\rightarrow&
3332221133&\rightarrow&1231231245\cr &&&&&&\text{\scriptsize column
word of}\cr  &&&&&&\text{\scriptsize
$\varrho_3(T)=\varrho^{BSS}(T)$}\end{array}$$ or
$$\begin{array}{cccccccccc}{\rm
T}={\tableau[Y]{~,~,1,1,1,1|~,1,2,2|2,3,3}}&\overset{\text{e}}\rightarrow&
{\rm
T}^e={\tableau[Y]{~,~,1,1,3,3|~,2,2,2|3,3,3}}&\overset{\text{$\bullet$}}{\rightarrow}&{\rm
T}^{e\bullet}={\tableau[Y]{~,~,~,1,1,1|~,~,2,2,2|1,1,3,3}}
&\overset{\text{$\blacklozenge$}}{\underset{\text{}}{\rightarrow}}
&{\rm
T}^{e\bullet\blacklozenge}={\tableau[Y]{~,~,1|~,1,2|1,2,3|2,3|4|5}}.
\end{array}$$\end{example}

%

\section{Computational  complexity of bijection $\blacklozenge$ and reduction of conjugation symmetry map}

\label{S:comp}

 We show that the computational complexity of bijection
$\blacklozenge$ is linear on the input. We follow closely~\cite{pak}
for this section. Using ideas and techniques of Theoretical Computer
Science, see~\cite{AHU, CLRS}, each bijection can be seen as an
algorithm having one type of combinatorial objects as \emph{input},
and another as \emph{output}. We define a \emph{correspondence} as
an one--to--one map established by a bijection; therefore, obviously
several different defined bijections can produce the same
correspondence. In this way one can think of a correspondence as a
function which is computed by the algorithm, viz. the bijection. The
computational complexity is, roughly, the number of steps in the
bijection. Two bijections are \emph{identical} if and only if they
define the same correspondence. Obviously one task can be performed
by several different algorithms, each one having its own
computational complexity, see~\cite{AHU, CLRS}. For example we recall that
there are several ways to multiply large integers, from naive
algorithms, e.g. the Russian peasant algorithm, to that ones using
FFT (Fast Fourier Transform), e.g. Sch\"{o}nhage--Strassen
algorithm; see e.g.~\cite{GG} for a comprehensive and update
reference. Formally, a function $f$ reduces linearly to $g$, if it
is possible to compute $f$ in time linear in the time it takes to
compute $g$; $f$ and $g$ are linearly equivalent if $f$ reduces
linearly to $g$ and vice versa. This defines an equivalence relation
on functions, which can be translated into a linear equivalence on
bijections.

 Let $D = \(d_{1}, \ldots , d_{n}\)$ be an array of
integers, and let $m = m\(D\) := \max_{i} d_{i}$. The
\emph{bit--size} of $D$, denoted by $\langle D \rangle$, is the
amount of space required to store $D$; for simplicity from now on we
assume that $\langle D \rangle = n \rf{\log_{2} m + 1}$. We view a
bijection $\delta : \cA \longrightarrow \cB$ as an algorithm which
inputs $A \in \cA$ and outputs $B = \delta\(A\) \in \cB$. We need to
present Young tableaux as arrays of integers so that we can store
them and compute their bit--size. Suppose $A \in YT\(\lambda / \mu;
m\)$: a way to encode $A$ is through its recording matrix
$\(c_{i,j}\)$, which is defined by $c_{i,j} =a_{i,j} - a_{i,j-1}$;
in other words, $c_{i,j}$ is the number of $j$'s in the $i$--th row
of $A$; this is the way Young tableaux will be presented in the
input and output of the algorithms. Finally, we say that a map
$\gamma : \cA \longrightarrow \cB$ is \emph{size--neutral} if the
ratio $\frac{\langle \gamma(A)\rangle}{\langle A \rangle}$ is
bounded for all $A \in \cA$. Throughout the paper we consider only
size--neutral maps, so we can investigate the linear equivalence of
maps comparing them by the number of times other maps are used,
without be bothered by the timing. In fact, if we drop the condition
of being size--neutral, it can happen that a map increases the
bit--size of combinatorial objects, when it transforms the input
into the output, and this affects the timing of its subsequent
applications. Let $\cA$ and $\cB$ be two possibly infinite sets of
finite integer arrays, and let $\delta : \cA \longrightarrow \cB$ be
an explicit map between them. We say that $\delta$ has linear cost
if $\delta$ computes $\delta\(A\) \in \cB$ in linear time
$O\(\langle A \rangle \)$ for all $A \in \cA$. There are many ways
to construct new bijections out of existing ones: we call such
algorithms \emph{circuits} and we define below several of them that
we need.

\begin{description}

\item{$\circ$}
Suppose $\delta_{1} : \cA_{1} \longrightarrow \cX_{1}$, $\gamma :
\cX_{1} \longrightarrow \cX_{2}$ and $\delta_{2} : \cX_{2}
\longrightarrow \cB$, such that $\delta_{1}$ and $\delta_{2}$ have
linear cost, and consider $\chi=\delta_{2} \circ \gamma \circ
\delta_{1} : \cA \longrightarrow \cB$. We call this circuit
\emph{trivial} and denote it by $I\(\delta_{1}, \gamma,
\delta_{2}\)$.

\item{$\circ$}
Suppose $\gamma_{1} : \cA \longrightarrow \cX$ and $\gamma_{2} : \cX
\longrightarrow \cB$, and let $\chi=\gamma_{2} \circ \gamma_{1} :
\cA \longrightarrow \cB$. We call this circuit \emph{sequential} and
denote it by $S\(\gamma_{1}, \gamma_{2}\)$.

\item{$\circ$}
Suppose $\delta_{1} : \cA \longrightarrow \cX_{1} \times \cX_{2}$,
$\gamma_{1} : \cX_{1} \longrightarrow \cY_{1}$, $\gamma_{2} :
\cX_{2} \longrightarrow \cY_{2}$, and $\delta_{1} : \cY_{1} \times
\cY_{2} \longrightarrow \cB$, such that $\delta_{1}$ and
$\delta_{1}$ have linear cost. Consider $\chi=\delta_{2} \circ
\(\gamma_{1} \times \gamma_{2}\) \circ \delta_{1} : \cA
\longrightarrow \cB$: we call this circuit \emph{parallel} and
denote it by $P\(\delta_{1},\gamma_{1},\gamma_{2},\delta_{2}\)$.

\end{description}

For a fixed bijection $\alpha$, we say that $\beth$ is an
$\alpha$\emph{--based ps--circuit} if one of the following holds:

\begin{description}

\item{$\bullet$}
$\beth=\delta$, where $\delta$ is a bijection having linear cost.

\item{$\bullet$}
$\beth=I\(\delta_{1}, \alpha, \delta_{2}\)$, where
$\delta_{1},\delta_{2}$ are bijections having linear cost.

\item{$\bullet$}
$\beth=P\(\delta_{1},\gamma_{1},\gamma_{2},\delta_{2}\)$, where
$\gamma_{1},\gamma_{2}$ are $\alpha$--based ps--circuits and
$\delta_{1},\delta_{2}$ are bijections having linear cost.

\item{$\bullet$}
$\beth=S\(\gamma_{1}, \gamma_{2}\)$, where $\gamma_{1}, \gamma_{2}$
are $\alpha$--based ps--circuits.

\end{description}

In other words, $\beth$ is an $\alpha$--based ps--circuit if there
is a parallel--sequential algorithm which uses only a finite number
of linear cost maps and a finite number of application of map
$\alpha$. The $\alpha$--cost of $\beth$ is the number of times the
map $\alpha$ is used; we denote it by $s\(\beth\)$.

Let $\gamma : \cA \longrightarrow \cB$ be a map produced by the
$\alpha$--based ps--circuit $\beth$. We say that $\beth$ computes
$\gamma$ at cost $s\(\beth\)$ of $\alpha$. A map $\beta$ is
\emph{linearly reducible} to $\alpha$, write $\beta \hookrightarrow
\alpha$, if there exist a finite $\alpha$--based ps--circuit $\beth$
which computes $\beta$. In this case we say that $\beta$ can be
computed in at most $s\(\beth\)$ cost of $\alpha$. We say that maps
$\alpha$ and $\beta$ are linearly equivalent, write $\alpha \sim
\beta$, if $\alpha$ is linearly reducible to $\beta$, and $\beta$ is
linearly reducible to $\alpha$. We recall, gluing together, results
proved in Section 4.2 of~\cite{pak}.

\begin{proposition}
Suppose $\alpha_{1} \hookrightarrow \alpha_{2}$ and $\alpha_{2}
\hookrightarrow \alpha_{3}$, then $\alpha_{1} \hookrightarrow
\alpha_{3}$. Moreover, if $\alpha_{1}$ can be computed in at most
$s_{1}$ cost of $\alpha_{2}$, and $\alpha_{2}$ can be computed in at
most $s_{2}$ cost of $\alpha_{3}$, then $\alpha_{1}$ can be computed
in at most $s_{1}s_{2}$ cost of $\alpha_{3}$. Suppose $\alpha_{1}
\sim \alpha_{2}$ and $\alpha_{2} \sim \alpha_{3}$, then $\alpha_{1}
\sim \alpha_{3}$ Suppose $\alpha_{1} \hookrightarrow \alpha_{2}
\hookrightarrow \ldots \hookrightarrow \alpha_{n} \hookrightarrow
\alpha_{1}$, then $\alpha_{1} \sim \alpha_{2} \sim \ldots \sim
\alpha_{n} \sim \alpha_{1}$.
\end{proposition}
We state now the computational complexity of bijection
$\blacklozenge$ and the reduction of conjugation symmetry map.
\begin{alg}\label{alg}{\em[Bijection $\blacklozenge$.]}
\qquad \\
\emph{Input:} LR tableau $T$ of skew shape $\lambda / \mu$, with
$\lambda=\(\lambda_{1} \geq \ldots \geq \lambda_{n}\)$, \\
$\mu=\(\mu_{1} \geq \ldots \geq \mu_{n}\)$, and filling
$\nu=\(\nu_{1} \geq \ldots \geq \nu_{n}\)$, having $A=\(a_{i,j}\)
\in M_{n \times n}\(\N\)$ \quad ($a_{i,j}=0$ if $j>i$) as (lower
triangular) recording matrix.

Write $\widetilde{A}$, a copy of the matrix $A$.

For $j:=n$ down to $2$ do

\quad For $i:=1$ to $n$ do

\quad \quad Begin

\quad \quad \quad If $i=j$ then
$\widetilde{a}_{i,i}:=\widetilde{a}_{i,i}+\lambda_{1}-\lambda_{i}$

\quad \quad \quad \phantom{If $i=j$} else

\quad \quad \quad \phantom{If $i=j$ else} If $j>i$ then
$\widetilde{a}_{i,j}=0$ else
$\widetilde{a}_{i,j}:=\widetilde{a}_{i,j}+\widetilde{a}_{i,j+1}$.

\quad \quad End

\bigskip

So far the computational cost is $O\(n^{2}\)=O\(\langle A
\rangle\)$.

\bigskip

Set a matrix $B=\(b_{i,j}\) \in M_{\lambda_{1} \times
\lambda_{1}}\(\N\)$ such that $b_{i,j}=0$ for all $i,j$.

For $i:=1$ to $n$ do

\quad Begin

\quad \quad Set $c:=0$.

\quad \quad For $j:=0$ to $n$ do

\quad \quad \quad Begin

\quad \quad \quad \quad $r:=\widetilde{a}_{i+j,i}-a_{i+j,i}$, \quad see Remark~\ref{R:comp}.

\quad \quad \quad \quad For $t:=1$ to $a_{i+j,i}$ do
$b_{r+t,c+t}:=b_{r+t,c+t}+1$.

\quad \quad \quad \quad $c:=c+a_{i+j,i}$.

\quad \quad \quad End

\quad End

This part has total computational cost at most equal to
$$O\(\sum_{1 \leq i.j \leq n}a_{i,j}\)=O\(|\lambda \setminus \mu|\)
=O\(|\lambda|-|\mu|\)=O\(\langle T \rangle\).$$

\emph{Output:} $B$ recording matrix of the output tableau.

\end{alg}

\begin{remark} \label{R:comp}
For all $1 \leq i \leq n$ and $0\leq j \leq n-i+1$,
we have
$$\widetilde{a}_{i+j+1,i}-\widetilde{a}_{i+j,i} \geq a_{i+j+1,i}.$$
\end{remark}
From Theorem \ref{reduc} and this algorithm we have
\begin{theorem}
The conjugation symmetry maps $\varrho^{BSS}$, $\varrho^{WHS}$ and $\varrho_{3}$ are
identical, and linear equivalent to the Sch\"utzenberger involution
$E$,
$$\begin{matrix}
T&\overset{\text{$e\,\bullet$}}\longleftrightarrow&T^{e\bullet}&\overset{\text{$\blacklozenge$}}\longleftrightarrow& {T^{e\bullet}}^\blacklozenge\\
{\tau}\updownarrow&&{\tau}\updownarrow\\
P&\underset{E}{\overset{\text{\scriptsize
evacuation}}\longleftrightarrow}&P^E.
\end{matrix}$$
Thus conjugation symmetry maps and commutative symmetry maps are linearly reducible to each other.
\end{theorem}

\noindent{\bf Acknowledgement:} The first author wishes to thank
 Igor Pak and Ernesto Vallejo for their interest in
the papers \cite{az,az1}.

\end{document}